\documentclass[11pt,a4paper, notitlepage]{article}
\usepackage[utf8]{inputenc}
\usepackage{colonequals}
\usepackage{graphicx,verbatim}
\usepackage{amssymb}
\usepackage{amsthm}
\usepackage{bbm}
\usepackage{amsmath}
\usepackage{mathtools}
\usepackage{textgreek}
\usepackage{mathbbol}
\usepackage{amsfonts}
\usepackage{mathrsfs}
\usepackage[english]{babel}
\usepackage{lscape} 
\usepackage[vcentermath,enableskew]{youngtab}
\usepackage{tikz,tikz-cd}
\usepackage{tabu}
\usetikzlibrary{arrows,positioning,automata,shadows,fit,shapes}
\usepackage{ stmaryrd }
\usepackage[english]{babel}
\usepackage{setspace}
\usepackage{microtype}
\usepackage[hang,flushmargin]{footmisc} 
\usepackage{hyperref}
\hypersetup{hypertexnames = false, bookmarksdepth = 2, bookmarksopen = true, colorlinks, linkcolor = black, citecolor = black, urlcolor = black, pdfstartview={XYZ null null 1}}

\newtheorem{theorem}{Theorem}[section]
\newtheorem{definition}[theorem]{Definition}
\newtheorem{example}[theorem]{Example}
\newtheorem{corollary}[theorem]{Corollary}
\newtheorem{lemma}[theorem]{Lemma}
\newtheorem{proposition}[theorem]{Proposition}
\newtheorem{remark}[theorem]{Remark}
\newtheorem*{theorem*}{Theorem}
\newtheorem*{definition*}{Definition}
\newtheorem*{lemma*}{Lemma}

\usepackage{chngcntr}
\counterwithin{table}{section}

\DeclareMathOperator{\im}{im}

\DeclareMathOperator{\sgn}{sgn}

\DeclareMathOperator{\Hom}{Hom}

\DeclareMathOperator{\ad}{ad}

\DeclareMathOperator{\Tr}{Tr}

\DeclareMathOperator{\Cent}{Cent}

\newcommand{\fsl}{\mathfrak{sl}}
\newcommand{\fso}{\mathfrak{so}}

\newcommand{\fosp}{\mathfrak{osp}}

\newcommand{\fg}{\mathfrak{g}}

\newcommand{\fo}{\mathfrak{o}}

\newcommand{\fA}{\mathfrak{A}}

\newcommand{\bbc}{\mathbb{C}}

\newcommand{\bbn}{\mathbb{N}}

\newcommand{\bbr}{\mathbb{R}}
\newcommand{\bbs}{\mathbb{S}}

\newcommand{\bbz}{\mathbb{Z}}

\newcommand{\Cc}{\mathcal{C}}

\newcommand{\Cf}{\mathcal{F}}
\newcommand{\Cg}{\mathcal{G}}

\newcommand{\Cs}{\mathcal{S}}

\newcommand{\Cv}{\mathcal{V}}
\newcommand{\Cw}{\mathcal{W}}

\newcommand{\bq}{\mathbf{q}}

\newcommand{\lpi}{\langle}
\newcommand{\rpi}{\rangle}

\newcommand{\Pin}{\mathsf{Pin}}
\newcommand\mygamma{{I} \hspace{-.65ex} \varGamma}
\newcommand{\D}{\mygamma} 
\newcommand{\rca}{\mathsf{H}}
\newcommand{\ama}{\mathsf{A}}
\newcommand{\tama}{O_{t,c}}
\newcommand{\tamaz}{O_{t,0}}
\newcommand{\centre}{Z}
\newcommand{\anZ}{\centre^{\textup{anti}}}
\newcommand{\grZ}{\centre^{\textup{gr}}}

\newcommand{\ugZ}{\centre^{\textup{ug}}}
\newcommand{\scasi}{\mathcal{S}}
\newcommand{\clif}{\mathcal{C}}
\newcommand{\C}{\omega} 
\newcommand{\rhoC}{\rho(\omega)}
\newcommand{\css}{\Omega_c} 
\newcommand{\oO}{\tilde\sigma} 
\newcommand{\rhow}{\rho({\tilde{w}})} 
\newcommand{\wlong}{\rho(\tilde{w}_0)} 
\newcommand{\hc}{\rca \clif} 

\title{The centre of the Dunkl total angular momentum algebra}
\author{Kieran Calvert, Marcelo De Martino, Roy Oste}
\date{}

\begin{document}

\maketitle

\begin{abstract} 
For a finite dimensional representation $V$ of a finite reflection group $W$, we consider the rational Cherednik algebra $\mathsf{H}_{t,c}(V,W)$ associated with $(V,W)$ at the parameters $t\neq 0$ and $c$. 
The Dunkl total angular momentum algebra $O_{t,c}(V,W)$ arises as the 
centraliser algebra of the Lie superalgebra $\mathfrak{osp}(1|2)$ containing a Dunkl deformation of the Dirac operator, inside the tensor product of $\mathsf{H}_{t,c}(V,W)$ and the Clifford algebra generated by $V$.

We show that, for every value of the parameter $c$, the centre of $O_{t,c}(V,W)$ is isomorphic to a univariate polynomial ring.
Notably, the generator of the centre changes depending on whether or not $(-1)_V$ is an element of the group $W$. Using this description of the centre, and using the projection of the pseudo scalar from the Clifford algebra into $O_{t,c}(V,W)$, we establish results analogous to ``Vogan's conjecture'' for a family of operators depending on suitable elements of the double cover $\tilde{W}$.

\end{abstract}

\section{Introduction}

The total angular momentum operators, which contain besides the orbital angular momentum also a spin angular momentum term, occur as symmetries of a Dirac Hamiltonian (see the introduction of \cite{DBOVJ18} for a brief overview).  
When one considers a Dunkl deformed version of a Dirac equation or operator, its symmetries form a deformation of the total angular momentum algebra. 
In this article, we continue the study of the Dunkl total angular momentum algebra for arbitrary real reflection groups. 
This is related to the theory of Howe dual pairs, which we will now explain. 

Let $(V_0,B_0)$ be a Euclidean pair with $V_0 \cong \mathbb{R}^d$ a real vector space and let $(V,B)$ be its complexification. 
Let also $\mathsf{O}(d) = \mathsf{O}(V,B) \subset GL(V)$ denote the orthogonal group of the pair $(V,B)$. 
Denote by $\mathcal{W} = \mathcal{W}(V)$ the Weyl algebra of polynomial coefficients partial differential operators acting on the polynomial space $\mathbb{C}[V]$. 
As is well known (see, e.g., \cite[Section 4, item (a)]{Ho89}), the Laplacian and its dual symbol, the squared norm, are $\mathsf{O}(d)$-invariant elements inside $\mathcal{W}$, and they generate a Lie algebra isomorphic to $\mathfrak{sl}(2,\mathbb{C})$. 
Moreover, the associative subalgebra generated by this realisation of $\mathfrak{sl}(2,\mathbb{C})$ gives all $\mathsf{O}(d)$-invariants in $\mathcal{W}$. 
The pair $(\mathsf{O}(d),\mathfrak{sl}(2))$ just described is one of the simplest examples in the theory of Howe dual pairs. 
Together, they give a multiplicity-free decomposition of $\mathbb{C}[V]$ in irreducible $(\mathsf{O}(d),\mathfrak{sl}(2))$-bimodules, where the linked $\mathsf{O}(d)$- and $\mathfrak{sl}(2)$-modules uniquely determine each other. 
The $\mathsf{O}(d)$-modules in this decomposition are precisely the spherical harmonics. 

In the Weyl-Clifford algebra $\mathcal{W}\clif = \mathcal{W}\otimes\clif$, where $\clif = \clif(V,B)$ is the Clifford algebra attached to $(V,B)$, the  $\mathsf{O}(d)$-invariants include square roots of the squared norm and the Laplacian, with the square root of the latter being the Dirac operator. 
These square roots generate a Lie superalgebra isomorphic to $\mathfrak{osp}(1|2,\mathbb{C})$, which contains the above-mentioned copy of $\mathfrak{sl}(2,\mathbb{C})$ as its even subalgebra. 
Also in this case, the associative subalgebra generated by this realisation of $\mathfrak{osp}(1|2,\mathbb{C})$ gives all $\mathsf{O}(d)$-invariants in $\mathcal{W}\clif$. 
Here, the relevant Howe dual pair is $(\mathsf{Pin}(d),\mathfrak{osp}(1|2,\mathbb{C}))$, to properly account for the spin-representations of $\mathsf{O}(d)$ occurring in a similar multiplicity-free decomposition and correspondence of irreducible modules, now of the space of spinor-valued polynomials (see \cite{Ni91}, \cite{CJHW10}, and also \cite{BDSES10} for classical dualities involving the Pin-group). 

Deformations of these Howe dual pairs occur when the action of the partial derivatives is replaced by the divided-difference operators introduced by Dunkl \cite{Du89}. In other words, we are interested in seeing the dual pairs of the previous paragraphs in the context of a rational Cherednik algebra $\rca_{t,c}(V,W)$ (see \cite{EG02} and \cite{DO03}) at a parameter function $c$, $t\neq 0$ and associated with $(V,W)$, where we now see $V$ as a (faithful) representation $W \hookrightarrow \mathsf{O}(V,B) \subset GL(V)$ of a real reflection group $W$ (in general, this action does not need to be irreducible nor essential). Using the well-known PBW properties of these algebras, there is a linear embedding $\mathcal{W} \hookrightarrow \rca_{t,c}(V,W)$ (respectively, $\mathcal{W}\clif \hookrightarrow \rca\clif_{t,c} := \rca_{t,c}(V,W)\otimes \clif$) and the image of $\mathfrak{sl}(2,\mathbb{C})$ (respectively, $\mathfrak{osp}(1|2,\mathbb{C})$) under this embedding still closes into a Lie algebra (respectively, Lie superalgebra) in the Cherednik context (see \cite{He91}, \cite{DBOSS12}). Since the full action of $\mathsf{O}(d)$ is not present in $\rca_{t,c}(V,W)$ (only $W\subseteq\mathsf{O}(d)$ acts on the space of Dunkl-operators), to make sense of dual pairs, one must first compute what is the centraliser algebra for the Lie (super)algebra in question. In \cite{CD20b} (see also \cite{BSO06}) the case of $\mathfrak{sl}(2)$ was considered and results on the joint-decomposition of the polynomial space for the action of the dual pair was obtained. Similar results for $\mathfrak{osp}(1|2)$ were obtained in \cite{DBOJ18}, while in \cite{Os21} the symmetry algebra was completely determined (see also \cite{OSS09} for an initial overview of the theory of deformations of the Howe dual pairs we consider in the Cherednik context, \cite{DBGV16} for the case when $W$ is a product of groups of type $A_1$, which is used to define a higher rank Bannai--Ito algebra, and \cite{DBOVJ18}, \cite{DBLROVJ22} for other specific choices of the reflection group $W$). 

In the present work, we focus on the $\mathbb{Z}_2$-graded algebra $\tama:=\tama(V,W)$ that is defined as the symmetry algebra of the above-mentioned realisation of $\mathfrak{osp}(1|2)$ inside $\rca\clif_{t,c}$. We shall refer to this algebra as the Dunkl total angular momentum algebra, since it contains the total angular momentum operators when $c=0$ and $t=1$. Our main result (Theorem \ref{t:TamaCent}) is a full description of the centre of $\tama$ for any real reflection group $W$ acting on $V$ and for any parameter function $c$. As an application to the determination of the centre, and inspired by the successful Dirac-theories for Drinfeld algebras \cite{Ci16} we prove in Theorem \ref{t:Vogan} and \ref{t:CentChar} results analagous to the celebrated ideas of Vogan on Dirac cohomologies (see \cite{Vo97}, \cite{HP02}, \cite{Ci16}). Our results build on Dirac theories for subalgebras of the Cherednik algebra (see \cite{Ca20} and \cite{CDM22}). The theory in this paper shares similarities with the  Dunkl angular momentum algebra \cite{CDM22}, in the sense that we consider a family of operators depending on certain ``admissible'' elements (see Definition \ref{d:Dirac_C}). We remark that in the present case, instead of enlarging the algebra in question with a suitably defined Clifford algebra and define a theory using a Dirac element, we do not tensor $\tama$ with another Clifford algebra  and we use a natural element inside the algebra $\tama$ itself (see Definition \ref{d:WGamma} -- classically when $c=0$, the eigenvalues of this element are precisely the square root of the total angular momentum quantum numbers), this reflects the theory for Hecke-Clifford algebras as in \cite{Ch17}.

Finally, we give a breakdown of the contents of the paper: in Section \ref{s:Prelims}, we recall the basic definitions of the rational Cherednik algebra, Clifford algebra and double cover of the reflection group. In this section, we also recall the precise realisation of the $\mathfrak{osp}(1|2)$ Lie superalgebra and we discuss the several notions of ``centre'' that we will consider in the context of $\bbz_2$-graded algebras. Next, in Section \ref{s:TAMA}, we define the algebra $\tama$ and we recall some structural properties such as the set of generators, based on recent results obtained in \cite{Os21}. In Section \ref{s:centre}, we prove our main result on the description of the centre. The key idea in our arguments is to compare with the result when $c=0$ by introducing a parameter $\bq$ and using generic versions of the Weyl and rational Cherednik algebras. Further, in Section \ref{s:Vogan}, we discuss the application to Vogan morphism and cohomologies in our context and in the last section, we discuss in detail the set of admissible elements of the group algebra $\bbc\tilde W$.

\section{Preliminaries}\label{s:Prelims}

We start this section by defining the bilinear products and variations of centre that we will use throughout this paper.  

\begin{definition}
Let $A = A_{\overline{0}}\oplus A_{\overline{1}}$ be a $\bbz_2$-graded associative algebra. Throughout this work we will be interested in 4 different bilinear products in $A$. The first is the associative multiplication of $A$ which will be denoted by juxtaposition. Given homogeneous elements elements $x,y\in A$ of degree $|x|$ and $|y|$ we denote by
\begin{align*}
    \llbracket x,y \rrbracket &= xy -(-1)^{|x||y|} yx\\
    [x,y] &= xy - yx\\
    \{x,y\} &= xy + xy.
\end{align*}
Accordingly, we define the \textbf{ungraded, graded} and \textbf{anti} centers, respectively denoted as $\ugZ(A), \grZ(A)$ and $\anZ(A)$, by
\begin{align*}
    \ugZ(A) &= \{z \in A \mid [ z,x ]=0 \textup{ for all } x\in A \},\\
    \grZ(A) &= \{z \in A \mid \llbracket z,x \rrbracket=0 \textup{ for all } x\in A \},\\
    \anZ(A) &= \anZ_{\overline{0}}(A) \oplus \anZ_{\overline{1}}(A).
\end{align*}    
where 
\begin{align*}
\anZ_{\overline{0}}(A) &= \{z \in A_{\overline{0}} \mid \{ z,x \} =0\textup{ for all } x\in A_{\overline{1}},[ z,x ] =0\textup{ for all } x\in A_{\overline{0}} \},\\
\anZ_{\overline{1}}(A) &= \{z \in A_{\overline{1}} \mid [ z,x ] =0\textup{ for all } x\in A\}.
\end{align*}
\end{definition}

\begin{remark}
Note that $\grZ(A) = \grZ_{\overline{0}}(A) \oplus \grZ_{\overline{1}}(A)$ where
\begin{align*}
\grZ_{\overline{0}}(A) &=  \{z \in A_{\overline{0}} \mid [ z,x ] =0\textup{ for all } x\in A\},\\
\grZ_{\overline{1}}(A) &=\{z \in A_{\overline{1}} \mid [ z,x ] =0\textup{ for all } x\in A_{\overline{0}}, \{ z,x \} =0\textup{ for all } x\in A_{\overline{1}}\},
\end{align*}
which justifies the terminology anti-center.
\end{remark}

\subsection{Rational Cherednik algebras}
Fix a vector space $V\cong \bbc^d$ with $d\geq 3$ and  a non-degenerate symmetric bilinear form $B$ on $V$. We view $B$ as the complexification of a Euclidean structure on $V_0\cong \bbr^d$. 
When needed, $\{y_1,\dotsc,y_d\}\subset V$ and $\{x_1,\dotsc,x_d\}\subset V^*$ will denote real orthonormal and dual bases, so that $\delta_{j,k} =\langle x_j,y_k\rangle =  B(y_j,y_k)$ for $j,k\in\{1,\dotsc,d\}$, where $\langle-,-\rangle:V^*\times V\to \bbc$ denotes the natural bilinear pairing.

Let $\mathsf{O}(d) \colonequals \mathsf{O}(V,B)$ denote the orthogonal group of the pair $(V,B)$ and  $ \mathsf{SO}(V,B)$ denote its identity component.  
We consider a finite real reflection group $W \subset \mathsf{O}(d)$. 
Let $R\subset V^*$ denote the root system of $W$ and fix a positive system $R^+ \subset R$, and a compatible choice of simple roots $\Delta$. Let $c\colon R^+ \to \mathbb{C}$ be a $W$-invariant parameter function and denote $c(\alpha) = c_\alpha$ where $c_{w\alpha} = c_\alpha$ for all $w\in W$ and $\alpha \in R$. 

For a root $\alpha\in R$, denote by $s_{\alpha}\in W$ the reflection in the hyperplane perpendicular to $\alpha$, and by $\alpha^\vee\in V$ the coroot such that $\langle \alpha^\vee,\alpha \rangle = 2$. Fix $t\in \bbc^\times$.

\begin{definition}\label{d:RCA}
Define $\rca_{t,c} \colonequals\rca_{t,c}(V,W)$ to be the quotient of $T(V \oplus V^*) \rtimes W$
by the following relations for $y,v \in V$ and $x,u\in V^*$:
\begin{equation}
	\label{e:RC}
		[x,u] = 0 = [y,v],  \qquad 
		[y,x] = t\langle y, x\rangle -   \sum_{\alpha>0} \langle y,	\alpha\rangle\langle  \alpha^{\vee},  x \rangle  c(\alpha) s_{\alpha} .
\end{equation} 		
\end{definition}

\begin{remark}
When $c=0$ and $t=1$, we have that $\rca_{1,0}=\Cw\rtimes W$, where $\Cw$ is the Weyl algebra of polynomial coefficient differential operators associated with $V\oplus V^*$.
Note that for a general $t\in \bbc^\times$, we have $\rca_{t,0} \cong \Cw\rtimes W$. We shall refer to the $c=0$ as the classical or undeformed case.
\end{remark}

\subsection{Clifford algebras}

We consider the Clifford algebra $ \clif \colonequals \clif (V,B)$ with canonical map $\gamma \colon V\to \clif $. Let $e_j \colonequals \gamma(y_j) $ for $j\in\{1,\dotsc,d\}$, then $ \clif $ is generated by $\{e_1,\dotsc,e_d\}$ satisfying 
	\begin{equation}\label{e:clifrel}
\{e_j,e_k\} 
=2 B(y_j,y_k) = 2 \delta_{jk}\,.
		\end{equation}
The Clifford algebra is naturally $\bbz_2$-graded with $\gamma(V)$ having degree $\overline{1}$. We shall denote by $\hc_{t,c}$ the tensor product $\rca_{t,c}\otimes \clif$.

For a subset  $A \subset \{1,\dotsc,d\}$, with elements $A= \{ i_{1},i_{2},\dotsc,i_{k}\} $ such that $1\leq i_{1}<i_{2}<\cdots <i_{k}\leq d$, we denote
			$
			e_A = e_{i_{1}}e_{i_
				{2}}\cdots e_{i_{k}}
			$.
Let $e_{\emptyset} = 1$, then a basis for $\clif$ as a vector space is given by $\{e_A \mid A \subset \{1,\dotsc,d\} \}$. 

We denote the chirality operator or pseudo-scalar of the Clifford algebra as
\begin{equation}\label{e:Gamma}
    \Gamma \colonequals i^{d(d-1)/2} e_1 \dotsm e_d \in \Cc;
\end{equation}
it satisfies $\Gamma^2 = 1$. 

In the Clifford algebra, there is a realisation of the group $\Pin \colonequals \Pin(V,B)$, which is a double covering of the orthogonal group $p \colon \mathrm{Pin} \to \mathsf{O}(d)$.
A double cover of the group $W$ is defined as $\tilde W = p^{-1}(W)$.
	
For a reflection $s \in W$ let $\tilde s$ denote a preimage in $\tilde W$, so $p(\tilde s) = s$. Let $\theta$ be the nontrivial preimage of $1$ in $\tilde{W}$. The element $\theta$ is central in $\tilde{W}$ and has order two: $\theta^2 =1$. The group $W$ has presentations
\[
W = \langle s_\alpha,\alpha\in R\mid s_\alpha^2=1,s_\alpha s_\beta s_\alpha = s_\gamma,\gamma = s_\alpha(\beta)\rangle,
\]
\[
W = \langle s_\alpha,\alpha\in \Delta \mid (s_\alpha s_\beta)^{m_{\alpha,\beta}} = 1 \rangle.
\]
While the double cover has presentations \textbf{}
\begin{equation}\label{e:PinPresentation}
\tilde{W} = \langle \theta, \tilde{s}_\alpha,\alpha\in R\mid \tilde{s}_\alpha^2=1=\theta^2,\tilde{s}_\alpha \tilde{s}_\beta \tilde{s}_\alpha = \theta\tilde{s}_\gamma,\gamma = s_\alpha(\beta), \theta \text{ central}\rangle,
\end{equation}
\begin{equation}\label{e:PinPresentationmalpha}
\tilde{W} = \langle \theta, \tilde{s}_\alpha,\alpha\in \Delta \mid (\tilde{s}_\alpha \tilde{s}_\beta)^{m_{\alpha,\beta}} = (\theta)^{m_{\alpha,\beta}-1}, \theta \text{ central}\rangle.
\end{equation}The group algebra $\bbc \tilde{W}$ splits into two subalgebras 
\begin{equation}\label{e:idemptildeW}
\bbc\tilde{W} = \frac{1}{2}(1+\theta) \bbc \tilde{W} \oplus \frac{1}{2}(1-\theta)\bbc\tilde{W}.
\end{equation}
we shall denote the algebras $\frac{1}{2}(1\pm \theta)\bbc\tilde{W}$ by, respectively $\bbc\tilde{W}_\pm$. The algebra $ \bbc W_+$ is isomorphic to $\bbc W$.
 We define a diagonal map $\rho$ from $\tilde{W}$ to $\hc_{t,c} = \rca_{t,c} \otimes \clif$:

\begin{equation}\label{e:rho}
	\rho \colon \tilde W \to 	\rca_{t,c} \otimes \clif  \colon \tilde w \mapsto p(\tilde w) \otimes \tilde w 
\end{equation}
which is extended linearly to a homomorphism on the group algebra $\mathbb{C}\tilde W$.

\begin{proposition}\label{p:rhoiso}
The image $\rho(\mathbb{C}\tilde W)$ is isomorphic to $\mathbb{C}\tilde W_-$.
\end{proposition}

\begin{proof}
Using the first isomorphism theorem it is sufficient to prove that the kernel of $\rho$ is $\frac{1}{2}(1+\theta) \bbc W_+$. 
Note that $\rho(\theta) = 1 \otimes -1$, therefore $\rho(1+\theta) = 0$ and hence $\bbc\tilde{W}_+$ is in the kernel. Furthermore, the image of $\rho$ contains the vectors $w \otimes p^{-1}(w)$, choosing a single element in $p^{-1}(w)$ for each $w$, one finds this space has $|W|$ linearly independent vectors. Therefore, the dimension of $\rho(\bbc \tilde{W})$ is equal to the dimension of $\bbc\tilde{W}_-$ and hence the full kernel is $\bbc\tilde{W}_+$.
\end{proof}

The tensor product $\rca_{t,c}\otimes \clif $ is $\bbz_2$-graded, inheriting the $\bbz_2$-grading from $\clif$.

\subsection{The realisation of \texorpdfstring{$\fosp(1|2)$}{osp(1|2)}}
In the undeformed case, the invariants for the action of $\mathsf{O}(d)$ in the Weyl-Clifford algebra $\Cw\otimes \clif$ are generated by the scalar products (the symmetric tensor corresponding to the bilinear form $B$, for a copy of $S^2(V^*)$ in $\Cw\Cc$), see \cite[Theorem 2.1 pg. 390]{Pr07} or \cite[Theorem 4.19]{CW12}. These elements form a realisation of the Lie superalgebra $\mathfrak{osp}(1|2)$ in $\Cw\otimes \clif$, and this realisation is preserved in the deformation $\rca\otimes \clif $. 
In terms of the $B$-orthonormal bases for $V$ and $V^*$, we can write these as 
	\begin{equation}\label{e:osp12}
	\begin{aligned}[]
		F^+ & 
		= 	\frac{1}{\sqrt{2t}}\sum_{p=1}^d x_p e_p ,& 
			F^- & 	
				= 	\frac{1}{\sqrt{2t}}\sum_{p=1}^d y_p e_p,\\
		E^+& 
		= \frac1{2t}\sum_{p=1}^d x_p^2,&
		E^-& 
		= -\frac1{2t}\sum_{p=1}^d y_p^2,\\
		H &
	= \frac{1}{t}\sum_{p=1}^d 
	\left(x_py_p + \frac{td}{2} - \Omega_c\right),
		\end{aligned}
\end{equation}
where the element $\css \colonequals \sum_{\alpha>0} c(\alpha)s_\alpha$ is central in $\bbc W$. These elements satisfy the relations
\begin{equation}\label{e:osp12re}
	\begin{aligned}[]
		\llbracket F^{+} , 	F^{-} \rrbracket & = H	,\quad
		&	\llbracket H , 	F^{\pm} \rrbracket &= \pm F^{\pm},
		&	
		\llbracket F^{\pm} , 	F^{\pm} \rrbracket & = \pm 2\,E^{\pm},
		\\	
		\llbracket E^{+} , 	E^{-} \rrbracket & = H,
		&	\llbracket H , 	E^{\pm} \rrbracket & = \pm2\,E^{\pm},	\quad	&
		\llbracket 	F^{\pm},E^{\mp}  \rrbracket & = 	F^{\mp}.	
	\end{aligned}
\end{equation}
The following elements in $U(\fosp(1|2))$ will play an important role in this manuscript.
\begin{definition}

The $\mathfrak{osp}(1|2)$ Scasimir operator is given by
\begin{equation}\label{e:Scasiosp}
\scasi = (F^-F^+ - F^+F^- - 1/2)  \in U(\fosp(1|2)),
\end{equation} 
 while the $\fosp(1|2)$ Casimir operator is given by
\begin{equation}\label{e:Casiosp}
		\Omega_\fosp = H^2
		+2(E^+E^-+E^-E^+)
		-(F^+  F^- -F^-F^+) \in U(\fosp(1|2)) .
\end{equation} 
\end{definition}
The Scasimir $\scasi$ is in the anti-centre of $U(\fosp(1|2))$ and the quadratic Casimir element $\Omega_\fosp$ is in the graded centre of $ U(\fosp(1|2))$. These two elements are related in the following well known way.
\begin{proposition}\label{p:scasisquare}
The Scasimir $\scasi$ squares to $\Omega_\fosp+\frac14$.
\end{proposition}
\begin{proof}
The above proposition is stated in \cite[Example 2, p. 9]{Fr96} with a different normalisation. 
\end{proof}
\begin{remark}

Note that $\Omega_{\fsl(2)} = H^2+2(E^+E^-+E^-E^+)$ is the quadratic Casimir element of the even subalgebra $\fsl(2)$ spanned by $H,E^+,E^-$. 
\end{remark}

\section{Centraliser algebra of \texorpdfstring{$\fosp(1|2)$}{osp(1|2)}}\label{s:TAMA}

\begin{definition}
The $\bbz_2$-graded algebra $\tama \colonequals O_{t,c}(V,W)$ is the graded centraliser of $\mathfrak{osp}(1|2)$, given by~\eqref{e:osp12} inside $\rca_{t,c}\otimes \clif $:
\[
O_{t,c}(V,W) \colonequals \{ \, a \in \rca\otimes \clif \mid  \llbracket a ,b \rrbracket = 0 \text{ for all } b \in \mathfrak{osp}(1|2) \,\} .
\]
\end{definition}

The elements of $\tama$ were described in~\cite{Os21}. We have that $\rho( \mathbb{C} \tilde W) \subset \tama$. 
Moreover, there is an isomorphism as $W$-modules ($\mathsf{O}(d)$-modules when $c=0$) from $\bigwedge (V)$ to a subspace of $\tama$, 
which sends $y_{i_{1}} \wedge y_{i_{2}} \wedge\cdots \wedge y_{i_{k}} \in \bigwedge^k( V)$, where   $A= \{ i_{1},i_{2},\dotsc,i_{k}\} \subset \{1,\dotsc,d\}$, to 
	\begin{equation}\label{e:OA}
	O_A = O_{i_{1}i_{2}\dotsm i_{k}} 
	= \bigg(\frac{|A| -1}{2}t  
		+ \sum_{a\in A} \oO_{a} e_{a}
		+\sum_{\{a,b\}\subset A } {M}_{ab} 
		e_{ab}	\bigg)e_A \in \tama,
\end{equation}
where $M_{ij} = x_iy_j - x_jy_i$ and the elements $\oO_j$ are defined as
\begin{equation}\label{e:Oj}
	\oO_j \colonequals \frac12\sum_{\alpha>0}  \langle y_j,\alpha \rangle \, c(\alpha) \,s_\alpha \,	\gamma({\alpha_s^{\vee}}) \in \rho( \mathbb{C} \tilde{W} ).
		\end{equation}
Note that $y_j\in \bigwedge^1 (V)$ is sent to $O_j = \oO_j$, and that an element of the form $O_{u_1\dotsm u_n} $ for $u_1,\dotsc,u_n\in V$ is skew-symmetric multilinear in its indices.

\begin{remark}
The notation $\oO_j$ is used instead of $O_j$ to emphasize that they are elements of $\rho( \mathbb{C} \tilde{W} )$, and also to more easily distinguish their occurrence in the algebra relations, section~\ref{s:rels}.
\end{remark}

\begin{remark}
When $A=  \{1,\dotsc,d\}$ in~\eqref{e:OA}, the element $O_{1\dotsm d}$ and the Scasimir $\scasi$ can be related by $\Gamma$ (see \cite[Section~3.1, page 1922]{DBOJ18} or \cite{Os21}),
\begin{equation}\label{e:SGamma}
\scasi \,\Gamma = \frac{i^{d(d-1)/2}}{t}O_{1\dotsm d}.
\end{equation}		
Furthermore, the square of the Scasimir can be written as (see \cite{Os21}),
\begin{equation}
\scasi^2 = \Omega_\fosp +\frac{1}{4} =  \frac{(d-1)(d-2)}{8} - \frac{(d-2)}{t^2}\sum_{j=1}^d (\oO_{j})^2  -\frac{1}{t^2}\sum_{1\leq j<k \leq d } ({O}_{jk})^2.
\end{equation}	
When $c=0$, the right-hand side is the quadratic Casimir of $\fso(d)$ (see Remark \ref{rem:classicalsod}, below). Its eigenvalues give the total angular momentum quantum numbers. 
\end{remark}
\noindent As a $\bbz_2$-graded algebra, $\tama$ is generated by (see \cite{Os21}) 
	\begin{itemize}
		\item $\rho(\tilde W)$, which has its usual $\mathbb{Z}_2$-degree,
		\item  the even elements 
			\begin{equation}\label{e:Oij}
				O_{ij} =  M_{ij}
				+t e_ie_j/2 + \oO_ie_j - \oO_je_i,
\end{equation}		
\item  and the odd elements 
		\begin{equation}\label{e:Oijk}
	O_{ijk} 	 = M_{ij} e_k -M_{ik} e_j +M_{jk} e_i + te_ie_je_k + \oO_i e_je_k-\oO_j e_ie_k+\oO_k e_ie_j.
	\end{equation}	
	\end{itemize}

In particular, we have~\cite{Os21}
	\begin{align*}	
	O_{klmn} & =  6 \, \mathcal A (O_{kl} O_{mn})- 8\, \mathcal A ( O_{klm}\oO_n)  \\
	& = \{O_{kl} ,O_{mn}\}- \{O_{km},O_{ln}\} + \{O_{kn},O_{ln}\} 
	\\
	& \quad -2( O_{klm}\oO_n- O_{kln}\oO_m+ O_{kmn}\oO_l - 	O_{lmn}\oO_k )\\
	O_{jklmn} & = 4 \mathcal A (O_{jkl} O_{mn})+ 48\mathcal A (O_{jkl} \oO_{m}\oO_{n}) - 36 \mathcal A (O_{jk} O_{lm}\oO_{n})
\end{align*}
where $\mathcal A$ denotes the antisymmetriser or antisymmetrizing operator, which has the following action on a multilinear expression with $n$ indices
\begin{equation}\label{e:asym}
\mathcal A (f_{u_1u_2\dotsm u_n})	= \frac{1}{n!} \sum_{s \in \mathrm{S}_n} \sgn (s) f_{u_{s(1)}} \dotsm \gamma_{u_{s(n)}}.
\end{equation}

\subsection{Relations}\label{s:rels}

In the algebra $\tama$, we have the following relations. These relations differ slightly from \cite{Os21} inasmuch as we define the Cherednik algebra with a parameter $t \neq 0$, here. For $\rho(\tilde w ) \in \mathbb{C} \tilde W_-$,
\begin{equation}
		\rho(\tilde w ) O_{u_1 \dotsb u_n} = (-1)^{|\tilde w|n} O_{p(\tilde w)\cdot u_1 \dotsb p(\tilde w)\cdot u_n} \rho(\tilde w ) .
\end{equation}
For $i,j,k,l,m,n$ distinct elements of the set $\{1,\dotsc,d\}$,
using the non-graded commutator $[A,B] = AB - BA$ and anticommutator $\{A,B\} = AB + BA$, we have the relations
	\begin{align}\label{e:15}
&	[O_{ij},\oO_{k}] -  [O_{ik},\oO_j] + [ O_{jk},\oO_i] = 0,\\  
& \{O_{ijk},\oO_l\}  - \{O_{ijl},\oO_k\} + \{O_{ikl},\oO_j\} - \{O_{jkl},\oO_i\} =  0,\label{e:16}
\end{align}	
\begin{align}
[O_{ij},O_{ki}] 
	&= tO_{jk} + [\oO_i,\oO_j]
	+\{O_{ijk},\oO_i\},\label{e:17}
	\\ 
	[O_{ij},O_{kl}]  
	&= 	\{\oO_i,O_{jkl}\}
	-\{\oO_j,O_{ikl}\},\label{e:18}
\end{align}	
	\begin{align}
	[O_{jk},O_{lmn} ]  & =  [\oO_j, O_{klmn}] - [\oO_k,O_{jlmn}],\label{e:19}\\
	[O_{jk},O_{jlm} ]  & = -tO_{klm} - \{\oO_k,O_{lm}\} - [\oO_j,O_{jklm}],\label{e:20}\\
	[O_{jk},O_{jkl} ]  & =  -\{\oO_j,O_{jl}\} - \{\oO_k,O_{kl}\}, \label{e:21}
\end{align}
	\begin{align}
\{O_{ijk},O_{ijk}\} & =  2\left( \oO_{i}^2+ 	\oO_{j}^2  + 	
								\oO_{k}^2 + O_{ij}^2 + O_{ik}^2 + 		O_{jk}^2\right) -\frac{t^2}{2}\label{e:22},\\ 
	\{O_{ijk},O_{ijl}\} & =  \{\oO_k,\oO_l\}
	    	  + \{O_{ik},O_{il}\}+ \{O_{jk},O_{jl}\},\label{e:23}\\
\{O_{ijk},O_{imn}\} & = t O_{jkmn}  	+ \{O_{jk},O_{mn}\}
	   	  +\{\oO_i,O_{ijkmn}\},	\label{e:24}\\
	\{O_{ijk},O_{lmn}\} 
 			& = 
			\{\oO_i,O_{jklmn}\} 
			- \{\oO_j ,O_{iklmn}\}
			+\{\oO_k , O_{ijlmn}\}. 	\label{e:25}
	\end{align}	

\begin{remark}\label{rem:classicalsod}
When $c=0$, the commutation relations~\eqref{e:17} and~\eqref{e:18} show that 
the linear span 
of the 2-index symmetries $O_{ij}$ forms a realisation of the Lie algebra $\fso(d)$. 
\end{remark}

\section{Centre of \texorpdfstring{$\tama$}{Oc}}\label{s:centre}

In order to determine the graded center of $\tama$, we shall first look at the classical graded center, that is, when $c=0$. 
For $c=0$, $\tamaz$ is realised inside $\hc_{t,0}=(\Cw\rtimes W)\otimes\clif$.

As a $\bbz_2$-graded algebra, the Weyl-Clifford algebra $\Cw\otimes \clif$ is generated by $\Cv = \Cv_{\bar0} \oplus \Cv_{\bar1}$ where $ \Cv_{\bar0} = V\oplus V^*$ and $\Cv_{\bar1}=V$. 
As an $\mathsf{O}(V,B)$-module, $\Cw\otimes \clif$ is isomorphic to the supersymmetric algebra $S(\Cv) = S(\Cv_{\bar0}) \otimes \bigwedge (\Cv_{\bar1})$, via the quantisation maps (see \cite[Proposition 5.4]{CW12}).

\begin{lemma}\label{lemma:Gamma}
The algebra of invariants $(\Cw\otimes\clif)^{\mathsf{O}(V,B)}$ is generated by the realisation of $\fosp(1|2)$ given by~\eqref{e:osp12}.
The algebra of invariants $(\Cw\otimes\clif)^{\mathsf{SO}(V,B)}$ is generated by $\fosp(1|2)$ and the Clifford algebra pseudo-scalar $\Gamma\in\clif$ given by~\eqref{e:Gamma}.
\end{lemma}
\begin{proof}
By \cite[Theorem 4.19]{CW12}, the invariants for $\mathsf{O}(V,B)$ in $S(\Cv)$ are generated by the quadratic invariants (the symmetric tensor corresponding to the bilinear form $B$, for a copy of $S^2(V^*)$ in $S^2(\Cv)$). 
In $\Cw\otimes\clif$, these can all be written in terms of the realisation of $\fosp(1|2)$ given in~\eqref{e:osp12}, and the constants. 
 
By \cite[Theorem 2.1 pg. 390]{Pr07}, the invariants for $\mathsf{SO}(V,B)$ are generated by the scalar products (the quadratic invariants) and the determinants (the alternating tensor corresponding to the map $\det$, for a copy of $\bigwedge^d (V^*)$ in $\bigwedge^d(\Cv)$). 
In $\Cw\otimes \clif$, the determinant tensors can all be written as products of quadratic invariants and the pseudo-scalar $\Gamma$ (which is the determinant tensor for the copy of $\bigwedge^d (V^*)$ inside $\bigwedge^d (V^*) \cong \clif$). 
\end{proof}







\begin{proposition}\label{p:ClassCenter}
When $c=0$, the graded center of $\tamaz$ in $\hc_{t,0}$ is the univariate polynomial ring in $\bbs$, where $\bbs$ is the Casimir $\Omega_{\fosp}$ of $\fosp(1|2)$ when $(-1)_V\notin W$, and $\bbs=\Cs(-1)_V$ with $\Cs$ the Scasimir of $\fosp(1|2)$ when $(-1)_V\in W$.
\end{proposition}
\begin{proof}
For $c=0$, the linear span $\fo := \langle O_{ij} \mid 1\leq i,j\leq d \rangle$ of the 2-index symmetries $O_{ij}$ forms a realisation of $\fso(d)$ inside $(\Cw\rtimes W)\otimes\clif$, as was noted in Remark \ref{rem:classicalsod}. 
For the subalgebra $\Cw\otimes\clif \subset (\Cw\rtimes W)\otimes\clif$, by exponentiation, we have $\Cent_{\Cw\otimes\clif}(\fo) = (\Cw\otimes \clif)^{\mathsf{SO}(V,B)}$ which is generated by $ \fosp(1|2)$ and $ \Gamma $, as given by Lemma~\ref{lemma:Gamma}.
From the action of $W$ on $\fo$, it follows that $\Cent_{(\Cw\rtimes W)\otimes\clif}(\fo)$ is generated by $ \fosp(1|2)$, $\Gamma $ and $W \cap \{1,(-1)_V\} $. 

Now, $ \fosp(1|2)$, and $W \cap \{1,(-1)_V\} $ supercommute with the elements $O_{ijk} \in \tamaz$ and $\rho(\tilde W)$. 
However, $\Gamma$ is in the anti-center, and not the graded center of the Weyl-Clifford algebra. Hence, $\Gamma$ does not supercommute with elements that have odd $\bbz_2$-grading, such as  $O_{ijk} \in \tamaz$. 
Since $\Gamma^2 = 1$, it follows that $\Cent_{(\Cw\rtimes W)\otimes\clif}(\tamaz )$ is generated by $ \fosp(1|2)$ and $W \cap \{1,(-1)_V\} $.

The claim now follows from intersecting $\Cent_{(\Cw\rtimes W)\otimes\clif}(\tamaz )$ with $\tamaz$: if $w_0 = (-1)_V$, then $\scasi (-1)_V = \scasi \,\Gamma\,\wlong \in \tamaz$, by~\eqref{e:SGamma}. \end{proof}
\color{black}

In order to determine the graded centre of the $O_{t,c}$ when $c\neq 0$, it is convenient to introduce a formal central parameter $\bq$ and define the Weyl and the rational Cherednik algebras as algebras over the polynomial ring $\bbc[\bq,\bq^{-1}]$. We will proceed in a similar fashion as in \cite{CD20b}. To that end, we define the generic Weyl algebra $\Cw_\bq$ as
the unital associative algebra over $\bbc$ generated by $\bq,\bq^{-1}$, $x\in V^*$ and $y\in V$ subject to the relations $[\bq^n,x] = 0 = [\bq^n,y] = [x,x'] = [y,y']$ for all $n\in \bbz$ and
\begin{equation}\label{eq:gradedWeyl}
[y,x] = \bq^2 \lpi y,x \rpi
\end{equation}
for all $x,x' \in V^*$ and $y,y' \in V$. In the next proposition, we will use the multiindex notation $x^\alpha = x_1^{a_1}\cdots x_d^{a_d}$ for any $\alpha = (a_1,\cdots,a_d)\in\bbn^d$.

\begin{proposition}\label{p:genWeyl}
The set $\{\bq^nx^\alpha y^\beta\mid n\in \bbz, \alpha,\beta\in\bbn^d\}$ forms a $\bbc$-linear basis of $\Cw_\bq$. Furthermore, if we define, for each $m\in \bbz$
\[
\Cw_\bq^m := \textup{span}_\bbc\{\bq^nx^\alpha y^\beta\mid n+|\alpha|+|\beta| = m\},
\]
then the generic Weyl algebra $\Cw_\bq = \oplus_{m\in\bbz} \Cw_\bq^m$ is a $\bbz$-graded $\bbc$-algebra.
\end{proposition}

\begin{proof}
The claim about the linear basis is immediate from the well-known linear isomorphism between the Weyl algebra and the symmetric algebra on $V\oplus V^*$. The description of the grading amounts to declaring $\bq, S^1(V\oplus V^*)$ to be of degree $1$ and $\bq^{-1}$ to be of degree $-1$. The result follows by observing that the defining relation (\ref{eq:gradedWeyl}) is a graded relation in $\Cw_\bq$.
\end{proof}

Similarly, we define the generic rational Cherednik algebra $\rca_{\bq,c}(V,W)$ by introducing the central parameter $\bq$ and requiring that the defining relation satisfy
\begin{equation}\label{eq:genRC}
 [y,x] = \bq^2\langle y, x\rangle - \sum_{\alpha>0} \langle y,	\alpha\rangle\langle  \alpha^{\vee},  x \rangle  c(\alpha) s_{\alpha} 
\end{equation}
for all $x\in V^*,y\in V$. Note that the non-generic Cherednik algebra of Definition \ref{d:RCA} can be obtained from the generic Cherednik algebra by sending $\textbf{q}^2$ to $t$.  By means of the well-known PBW linear basis of $\rca_{t,c}(V,W)$, it is straightforward to check that monomials of the type $\bq^nx^\alpha y^\beta w$, with $n\in \bbz, \alpha,\beta$ multi-indices and $w\in W$ form a linear basis of the generic rational Cherednik algebra. This $\bbc$-linear basis is independent of the choice of the parameter function $c$, and note that when $c=0$, we get $\rca_{\bq,0}(V,W) = \Cw_\bq \rtimes W$. Now define a filtration $\Cf^{(m)}$ on $\rca_{\bq,c}(V,W)$ in the following way. We declare $\bq,S^1(V+V^*)$ to be of degree $1$, $w \in W$ to be of degree $0$ and $\bq^{-1}$ to be of degree $-1$. Let
\begin{equation}\label{eq:filtration}
\Cf^{(m)} = \textup{span}_\bbc\left\{\bq^n x^\alpha y^\beta w  \mid n+|\alpha|+|\beta| \leq m\right\}.
\end{equation}

\begin{proposition}\label{p:bracket}
Given any $\xi\in\Cf^{(m)},\eta\in \Cf^{(n)}$ we have that
\[
[\xi,\eta] \equiv [\xi,\eta]_0
\]
modulo $\Cf^{(m+n-1)}$, where $[\xi,\eta]_0$ denotes the commutator product in the algebra $\rca_{\bq,0}(V,W) = \Cw_\bq \rtimes W$ at $c=0$.
\end{proposition}

\begin{proof}
It suffices to prove the result when $\xi$ is a monomial in $S^m(V+V^*)$, since $[\bq^n\xi w,\eta] = \bq^n(\xi[w,\eta] + [\xi,\eta]w)$ and $[w,\eta] = [w,\eta]_0$, for all $w \in W$. Let $p\in S(V^*)$ and $q\in S(V)$. For any $y\in V, x\in V^*$, it is known that (see \cite{Gr10} and \cite[Propositions 2.5, 2.6]{CD20a})
\begin{align*}
    [y,p] &= \bq^2\partial_y(p) - \sum_{\alpha >0} c(\alpha) \langle \alpha,y \rangle \frac{p-s_\alpha(p)}{\alpha}s_\alpha\\
    [q,x] &= \bq^2\partial_x(q) - \sum_{\alpha >0} c(\alpha) \langle x,\alpha^\vee \rangle \frac{q-s_\alpha(q)}{\alpha^\vee}s_\alpha.
\end{align*}
Hence, the claim holds when $\xi\in S^1(V+V^*)$. Now, given any $\nu \in S^{1}(V+V^*)$ and $\xi\in S^m(V+V^*)$, from
$[\nu\xi,\eta] = \nu[\xi,\eta] + \xi[\nu,\eta]$, the result is proved by induction on the monomial degree.
\end{proof}

\begin{corollary}\label{c:genRCA}
Let $\mathsf{Gr}(\rca_{\bq,c}(V,W))$ be the associated graded algebra with respect to the filtration defined in (\ref{eq:filtration}). Then, as $\bbz$-graded $\bbc$-algebras, we have $\mathsf{Gr}(\rca_{\bq,c}(V,W))\cong \rca_{\bq,0}(V,W)=\Cw_\bq\rtimes W$.
\end{corollary}

Now let $\hc_{\bq,c} = \rca_{\bq,c}(V,W)\otimes \clif$. We define a filtration $\Cg^{(m)}$ on $\hc_{\bq,c}$ similar to the filtration $\Cf^{(m)}$ of (\ref{eq:filtration}), but requiring that the Clifford elements are of degree $0$.

\begin{corollary}\label{c:GrgenRCA}
When $c=0$, the algebra $\hc_{\bq,0} = (\Cw_\bq\rtimes W)\otimes \clif$ is a $\bbz$-graded $\bbc$-algebra. For any $c$, with respect to the filtration $\Cg^{(m)}$, the associated graded object $\mathsf{Gr}(\hc_{\bq,c})$ is isomorphic to $\hc_{\bq,0}$ as $\bbz$-graded $\bbc$-algebras.
\end{corollary}

\begin{proof}
The set \[\{\bq^nx^\alpha y^\beta w \otimes e_A\mid n\in \bbz,w \in W, A \subset \{1,\ldots,d\},\alpha,\beta\textup{ multiindices}\}\] is a $\bbc$-linear basis of $\hc_{\bq,c}$, for all $c$. Furthermore, the product of two such monomials $\mu_1=\bq^{n_1}x^{\alpha_1} y^{\beta_1} w_1 \otimes e_{A_1}$ and $\mu_2=\bq^{n_2}x^{\alpha_2} y^{\beta_2} w_2 \otimes e_{A_2}$ can be written as
\[
\mu_1\mu_2 = \bq^{n_1+n_2}x^{\alpha_1}(w_1(x^{\alpha_2})y^{\beta_1}+[y^{\beta_1},w_1(x^{\alpha_2})])y^{\beta_1}w_1(y^{\beta_2}) w_1w_2 \otimes e_{A_1}e_{A_2}.
\]
The filtration degree of such expression depends on the commutator in $\rca_{\bq,c}(V,W)$, so our claims follow from Proposition \ref{p:genWeyl} and Corollary \ref{c:genRCA}.
\end{proof}

Next, let $\fg\subset \rca_{t,c}(V,W)\otimes\clif$ denote the realisation of the $\mathfrak{osp}(1|2)$ Lie superalgebra of (\ref{e:osp12}) and let $\fg_{\overline{0}}$ and $\fg_{\overline{1}}$ denote the even and odd parts of $\fg$. Denote by $\ama_{t,c} = \Cent_{\rca_{t,c}}(\fg_{\bar{0}})$.

Slightly abusing the notation, we still denote by $\fg$ the $5$-dimensional vector subspace of $\hc_{\bq,c}$ spanned by elements in (\ref{e:osp12}) defined by substituting $\bq=\sqrt{t}$. Note that when $c=0$, $\fg$ is a Lie superalgebra concentrated in degree $0$ inside $\hc_{\bq,c}$ and we still denote by $\fg_{\overline{0}}$ and $\fg_{\overline{1}}$ to the even and odd parts (but we remark that the $\bbz_2$-grading of $\fg$ is not compatible with the $\bbz$-grading of $\hc_{\bq,c})$. 

Let $\ama_{\bq,c} = \Cent_{\rca_{\bq,c}}(\fg_{\overline{0}})$ and $O_{\bq,c} = \Cent_{\hc_{\bq,c}}(\fg)$. It is straight forward to check that $\Cent_{\hc_{\bq,c}}(\fg_{\bar0}) = \ama_{\bq,c}\otimes \clif$ and that the following assertions hold true (the equivalent proofs in \cite{Os21} generalise to include $\textbf{q}$ in a straight forward way):
\begin{itemize}
    \item $\Cent_{\hc_{\bq,c}}(\fg) = P(\ama_{\bq,c}\otimes \clif)$,  where $P = \textup{Id} - \ad(F^-)\ad(F^+)$ is the projection operator $P:\Cent_{\hc_{\bq,c}}(\fg_{\overline{0}})\to \Cent_{\hc_{\bq,c}}(\fg)$,
    \item $\rho(\tilde W)$ and the elements $O_A = -\tfrac{\bq^2}{2}P(e_A)$, with $A\subseteq \{1,\ldots, d\}$, generate $O_{\bq,c}$ as an associative algebra over $\bbc[\bq,\bq^{-1}]$.
\end{itemize}
    
\begin{proposition}\label{p:tamaiso}
When $O_{\bq,c}$ is equipped with the filtration induced by the filtration $\Cg^{(m)}$ on $\hc_{\bq,c}$ as in Corollary \ref{c:GrgenRCA}, we have $\mathsf{Gr}(O_{\bq,c})\cong O_{\bq,0}$ is an isomorphism of $\bbz$-graded $\bbc$-algebras.
\end{proposition}

\begin{proof}
We note that for any $A\subset\{1,\ldots,d\}$ we have
\begin{align*}
   -\frac{\bq^2}{2}P(e_{A}) &= O_{A}\\
   &=\left( \sum_{\{a,b\} \subset A}M_{ab}e_{ab}  + \sum_{a \in A} \oO_ae_a + \frac{\bq^2(|A|-1)}{2} \right) e_A\\
   &\equiv \left( \sum_{\{a,b\} \subset A}M_{ab}e_{ab}  +  \frac{\bq^2(|A|-1)}{2} \right) e_A
\end{align*}
modulo $\Cg^{(1)}$. The algebras $O_{\bq,c}$ and $O_{\bq,0}$ are in $\hc_{\bq,c}$ and $\hc_{\bq,0}$, respectively, and from Corollary \ref{c:GrgenRCA}, we have $\mathsf{Gr}(\hc_{\bq,c}) \cong \hc_{\bq,0}$. Therefore, $\mathsf{Gr}(O_{\bq,c})$ is a subalgebra of $\hc_{\bq,0}$. The equation above shows that the subalgebras $\mathsf{Gr}(O_{\bq,c})$ and $O_{\bq,0}$ of $\hc_{\bq,0}$ coincide, from which we conclude $(4)$.
\end{proof}

\begin{theorem}\label{t:TamaCent}

The graded centre of $\tama$ is the polynomial ring $\bbc[\bbs]$, where $\bbs =\Omega_\fosp$ if $w_0 \neq(-1)_V$ and $\bbs =\Cs w_0$ if $w_0 = (-1)_V$.
\end{theorem}

\begin{proof}
From Proposition \ref{p:ClassCenter} the centre $Z^{\textup{gr}}(\tamaz)$ is generated by $\bbs$. 
We argue inclusion in both directions to show $Z^{\textup{gr}}(\tama) \cong Z^{\textup{gr}}(\tamaz)$. Any polynomial in $\bbs$ is central in $\tama$ and powers of $\bbs$ are linearly independent. Thus, there is an injective map from $Z^{\textup{gr}}(\tamaz)$ to $Z^{\textup{gr}}(\tama)$. 

Any element in the centre $Z^{\textup{gr}}(O_{\bq,c})$ must be such that $[z,a] = 0$ for all $a \in \tama$. Using Proposition \ref{p:bracket}, $[z,a]_0 = 0$ for all $a \in \tamaz$. Hence, $z$ is in the centre of the associated graded algebra. Because $\mathsf{Gr}(O_{\bq,c})$ is isomorphic to $O_{\bq,0}$, the centre $Z^{\textup{gr}}\mathsf{Gr}(O_{\bq,c})$ is isomorphic to $Z^{\textup{gr}}(O_{\bq,0})$. Therefore, we have the inclusion $Z(O_{\bq,c}) \subset Z^{\textup{gr}}(O_{\bq,0})$. Specialising $\bq$ to $\sqrt{t}$ proves that $Z^{\textup{gr}}(\tama) \subset Z(\tamaz)$.
\end{proof}

\begin{corollary}
The projection map $P = \textup{Id} - \ad(F^-)\ad(F^+)$ is a vector space isomorphism between $Z(\ama_{t,c})$ and $\grZ(\tama)$.
\end{corollary}








\begin{proof}
Recall that $w_0$ denotes the longest element of $W$. In \cite{CDM22,FH15} (see also \cite[Remark 3.3]{FH22}) it was proved that $Z(\mathsf{A}_{t,c})$ is the univariate polynomial ring $\mathcal{R}[\Omega_{\mathfrak{sl}(2)}]$, where $\mathcal{R}=\mathbb{C}$ if $ (-1)_V$ is not in $W$ and  $\mathcal{R} = \mathbb{C}[w_0]$ if $w_0=(-1)_V$ is in $W$. One computes
\begin{equation}\label{e:projScasi}
P(\scasi) = (-2)( \Omega_{\mathfrak{osp}} + \tfrac{1}{4}),
\end{equation}
so that using $\Omega_{\mathfrak{sl}(2)} = \Omega_{\mathfrak{osp}} - \scasi -\tfrac{1}{2}$, one obtains
\begin{equation}\label{e:projslcasi}
P(\Omega_{\mathfrak{sl}(2)}) = 3 \Omega_{\mathfrak{osp}}.
\end{equation}
Furthermore, when $w_0 = (-1)_V$
\begin{equation}\label{e:projlongest}
P(w_0) = (-2)\scasi w_0. 
\end{equation}
So from Theorem \ref{t:TamaCent}, in any case, the generators of $Z(\mathsf{A}_{t,c})$ are sent to generators of $\grZ(\tama)$. However, the projection operator is not an algebra homorphism, when restricted to $Z(\mathsf{A}_{t,c})$.  Notwithstanding, we claim that, for all $m\in\mathbb{Z}_{\geq 1}$, there exists $a_m\neq 0$ and a polynomial $q_{m-1}\in\bbc[\Omega_{\mathfrak{osp}}]$ of degree strictly smaller than $m$  such that $P(\Omega_{\mathfrak{sl}(2)}^m) = a_m\Omega_{\mathfrak{osp}}^m + q_{m-1}$. Indeed, the base case $m=1$ is (\ref{e:projslcasi}) with $a_1=3$. Assuming it holds true for $m$, note that
\begin{equation}\label{e:mpower}
\Omega_{\mathfrak{sl}(2)}^{m+1} =  (\Omega_{\mathfrak{osp}} - \scasi -\tfrac{1}{2})\Omega_{\mathfrak{sl}(2)}^{m} = \Omega_{\mathfrak{osp}}\Omega_{\mathfrak{sl}(2)}^{m} - \scasi\Omega_{\mathfrak{sl}(2)}^{m} -\tfrac{1}{2}\Omega_{\mathfrak{sl}(2)}^{m}.
\end{equation}
The important property we shall use is that that $P(ST) = P(S)T$, whenever $T$ is already an element of $\tama$. Now, since $\scasi+\tfrac{1}{2}$ commutes with $\Omega_{\mathfrak{osp}}$, we can use the binomial formula to expand $\Omega_{\mathfrak{sl}(2)}^{m} =( \Omega_{\mathfrak{osp}} - (\scasi + \tfrac{1}{2}))^m$. Using (\ref{e:projScasi}), we get
\begin{equation}\label{e:Smpower}
P(\scasi\Omega_{\mathfrak{sl}(2)}^{m}) = (-2)\Omega_{\mathfrak{osp}}^{m+1} + p_{m},
\end{equation}
where $p_m$ is a polynomial on $\Omega_{\mathfrak{osp}}$ of degree at most $m$. Using (\ref{e:mpower}),  (\ref{e:Smpower}) and the inductive hypothesis, we get
\[
P(\Omega_{\mathfrak{sl}(2)}^{m+1}) = a_{m+1}\Omega_{\mathfrak{osp}}^{m+1} + q_{m}
\]
with $a_{m+1} = a_m + 2$, proving our claim. We thus conclude that $P$ maps $\bbc[\Omega_{\mathfrak{sl}(2)}]$ isomorphically into $\bbc[\Omega_{\fosp}]\subseteq \grZ(O_{t,c})$, as a linear map. This settles the proof in the case when $(-1)_V$ is not in $W$.

Now suppose that $w_0 = (-1)_V$. The above argument shows that $P$ induces a linear isomorphism $\bbc[\Omega_{\mathfrak{sl}(2)}] \cong \bbc[\bbs^2]$. To conclude our proof, we need to show that $P$ maps $\Omega_{\mathfrak{sl}(2)}^mw_0$ to $a_m\bbs^{2m + 1} + q_{m}$, with $a_m\neq 0$ and $q_{m}$ a polynomial on $\bbs = \Cs w_0$ of degree\footnote{We can conclude that $q_m$ is an \emph{odd} polynomial on $\bbs$, but this is not essential in this proof.} at most $2m$. We use, once again, the binomial formula to expand
$\Omega_{\mathfrak{sl}(2)}^{m}w_0 =( \Omega_{\fosp} - (\scasi + \tfrac{1}{2}))^mw_0 = \Omega_{\fosp}^mw_0 + p_m$, where we can interpret $p_m$ as a polynomial on $\Omega_{\fosp},\Cs$ and $w_0$. Noting that $\Omega_{\fosp},\bbs = \Cs w_0\in \tama$, we apply $P$ to $(\Omega_{\mathfrak{sl}(2)}^{m}w_0)$. From $\Omega_\fosp^m = \bbs^{2m} + r_m$ (with $r_m\in \bbc[\bbs]$ of degree less than $2m$) and (\ref{e:projlongest}), the result follows.
\end{proof}
\section{Analogue of the Vogan morphism}\label{s:Vogan}
 
The rational Cherednik algebra $\rca_{t,c}$ is endowed with an anti-involution $\ast$ defined as follows. 
\begin{equation}\label{eq:bullet}
    w^\ast = w^{-1},  \quad x_i^\ast = y_i, \quad y_i^\ast = x_i, 
\end{equation}
for all $ w \in W$,
with $\{y_1,\ldots y_d\}$ and $\{x_1,\ldots,x_d\}$ any fixed pair of dual bases for $V$ and $V^*$. 
We define an anti-involution on $\clif$, $\gamma^*= (-1)^{|\gamma|}\gamma^t$. Here, if $\gamma = \eta_1 \dotsm \eta_p$ then $\gamma^t= \eta_p \dotsm \eta_1$. We extend these anti-involutions  to $\rca_{t,c} \otimes \clif$ by defining $\bullet: \rca_{t,c} \otimes \clif \to \rca_{t,c} \otimes \clif$ where $\bullet = \ast \otimes *$. The algebra $\tama$ inherits the anti-involution $\bullet$ from $\rca_{t,c} \otimes \clif$.

\begin{definition}\label{d:WGamma}
We define the element $\D$ by 
\[
\D  = \Gamma \scasi \in \tama.
\]
\end{definition}

\begin{remark}
In the classical case $c =0$ and trivial $W$, the eigenvalues of the element $\D$ acting on the appropriate polynomial-spinor space is a square root of the total angular momentum quantum number. 
\end{remark}
\begin{remark}
In \cite{Os21}, a projection operator $P$ (which we use in Proposition \ref{p:tamaiso}) is defined from the centraliser of $\fsl_2$ to $\tama$. 
We have that $\D = -P(\Gamma)/2$, the projection of the chirality (or pseudo-scalar) element~\eqref{e:Gamma}. 
For this reason we shall refer to $\D$ as the projected chirality operator. 
\end{remark}

\noindent Because the projection $P$ takes $\Cent_{\hc_{t,c}}(\fg_{\overline{0}})$ to $\tama$ and $\Gamma$ is in $\clif \subset \Cent_{\hc_{t,c}}(\fg_{\overline{0}})$, the projected chirality operator $\D$ is in $\tama$.

\begin{proposition}
The element $\D$ is self adjoint, that is
\[ 
{\D}^\bullet = \D.
\]
\end{proposition}

\begin{proof}
Note that  $\Gamma^\bullet= \Gamma^* = \Gamma$. Furthermore $(F^\pm)^\bullet = -F^\mp$. Hence, 
\[
\scasi^\bullet = \left(F^-F^+- F^+F^- - \frac{1}{2}\right)^\bullet = \left( (F^+)^\bullet(F^-)^\bullet-(F^-)^\bullet(F^+)^\bullet -\frac{1}{2} \right) = \scasi.
\]
The projected chirality operator is a product of two commuting self adjoint operators and therefore, is self adjoint. 
\end{proof}

\begin{proposition}\label{p::Dsquare}
In $\tama$, the element $\D$ is a square root of the Casimir $\Omega_{\fosp}$,
\[ 
(\D)^2 = \Omega_{\fosp} + \frac{1}{4}.
\]
\end{proposition}
\begin{proof}
The square of $\D$ is equal to the square of $\scasi$. The proposition follows from Proposition \ref{p:scasisquare} which is the equivalent statement for the Scasimir $\scasi$ element of $\fosp(1|2)$.
\end{proof}
We define a function $\epsilon: \bbc\tilde{W}_- \to \{\pm 1\}$ such that, for every homogeneous $\rhow \in \bbc\tilde{W}_-$,
\[ 
 \D \rhow= \epsilon (\rhow)  \rhow\D.
\]
If the dimension $d$ of $V$ is odd then $\epsilon(\rhow) = 1$ for all $\rhow \in \bbc\tilde{W}_-$. Alternatively, if $d$ is even then $\epsilon(\rhow) =  (-1)^{|\rhow|}$ for $\rhow\in \bbc\tilde{W}_-$, where $|\rhow|$ is the $\bbz_2$-grading of $\rhow$.

\begin{definition}
We define the $\epsilon$-centre of $\bbc\tilde{W}_-$ to be: 
\[
Z^\epsilon(\bbc\tilde{W}_-) = \{ a \in \bbc\tilde{W}_- : a b = \epsilon(b) b a,
\text { for all } b \in \bbc\tilde{W}_-\}.
\]
Furthermore, we say an element is $\epsilon$-central if it is contained in the $\epsilon$-centre.
\end{definition}

\noindent Since $\epsilon$ is valued in $\{-1,1\}$ then if two elements $\C,\nu$ are $\epsilon$-central then their product $\C\nu$ is central (in the ungraded sense). Furthermore if both $\C,\nu$ have the same $\bbz_2$-degree, then $\C\nu$ is even and also central, in the graded sense. 

\begin{definition}\label{d:Dirac_C}
A homogenous element $\C\in \bbc\tilde{W}_-$ is called {\bf admissible} if $\C$ is $\epsilon$-central and $
\C^\bullet =  \C$. We will denote by $\fA = \fA(\bbc\tilde{W}_-)$ the set of admissible elements. For any admissible $\C\in \fA$, define
\begin{equation}\label{e:CDirac}
\D_\C := \D + \rhoC \in \tama.
\end{equation}

\end{definition}

\begin{lemma} The square of $\D_\C$ can we written as;
\[
(\D_\C)^2 = \Omega_\fosp + \rhoC^2 + (1+\epsilon(\rhoC)) \rhoC \D + \frac{1}{4}.
\]
\end{lemma}

\begin{proof}
The following calculation uses Proposition  \ref{p::Dsquare} to calculate the square of $\D_\C$.
\begin{equation}
\begin{aligned}
   (\D_\C)^2 &= (\D + \rhoC)^2 \\
   &= (\D)^2 + \rhoC^2 + \D \rhoC + \rhoC \D \\
            &= (\D)^2 + \rhoC^2 + \epsilon(\rhoC)\rhoC\D  + \rhoC \D \\
            &= \Omega_\fosp + \frac{1}{4} + \rhoC^2 + (\epsilon(\rhoC)+1)\rhoC\D,
\end{aligned}
\end{equation}
finishing the proof.
\end{proof}
\begin{remark}
In the equation for $\D_\C^2$ above the element $\Omega_\fosp$ is central in $\tama$ and $\rho(\C)^2$ is central in $\bbc\tilde{W}_-$. 
\end{remark}

We now prove an analogue of Vogan's conjecture for the algebra $\tama$. Thus, for every choice of $\C$ we can relate the centre of $\tama$ with the centre of the algebra $\ugZ(\bbc\tilde{W}_-)$. This will allow us, once we have constructed the $\D_\C$-cohomology (Definition \ref{d:DiracCoh}), to relate the action of the centre of the these algebras on  $\tama$-modules if the cohomology is non-zero.

\begin{theorem}\label{t:Vogan}
Given an admissible $\C \in \bbc\tilde{W}_-$, there is an algebra homomorphism 
\[
\zeta_\C: \grZ(\tama) \to \ugZ(\bbc\tilde{W}_-)
\]
such that, for all $z \in Z(\tama)$ there exists $a\in \tama$ satisfying
\begin{equation}\label{e:zetaC}
z = \zeta_\C(z) + \D_\C a + a\D_\C.
\end{equation}
\end{theorem}

\begin{proof} The proof here employs identical ideas to the proof in \cite[Theorem 5.4]{CDM22}. By Theorem \ref{t:TamaCent} the centre of $\tama$ is polynomial in the element 
\[
\bbs = \begin{cases} \Omega_\fosp & \text{ if } w_0 \neq (-1)_V, \\ \D\wlong & \text{ if } w_0 = (-1)_V. \end{cases}
\]
Furthermore, we have shown that 
\begin{align*}
\Omega_\fosp &= (\D_\C)^2 +\rhoC^2 - \frac{1}{4} -\D_\C \rhoC -\rhoC\D_\C\\
&= \D_\C (\frac{1}{2} \D_\C - \rhoC) + (\frac{1}{2}\D_\C - \rhoC)\D_\C + \rhoC^2 - \frac{1}{4}.
\end{align*}
 We split the proof, depending on $w_0$. If $w_0 \neq (-1)_V$, we conclude that if $\zeta_\C$ exists it must be such that $\zeta_\C(\Omega_\fosp) = \rhoC^2 - \frac{1}{4}$. 
We define an algebra homomorphism $\zeta_\C: \grZ(\tama) \to \ugZ(\bbc\tilde{W}_-)$ by $\zeta_\C(\Omega_\fosp^k) = (\rhoC^2- \frac{1}{4})^k$ and extend linearly. We will prove that $\zeta_\C$ satisfies condition (\ref{e:zetaC}). Since this condition is linear and $\Omega_\fosp^k$ is a basis for $\grZ(\tama)$, we are left to prove that there exists an $a_k \in \tama$ such that 
\[ 
\Omega_\fosp^k = \D_\C a_k + a_k\D_\C + \left(\rhoC^2-\frac{1}{4}\right)^k.
\]
We prove this by induction, let $a_{k-1}$ be such that $\Omega_\fosp^{k-1} = \D_\C a_{k-1} + a_{k-1}\D_\C + (\rhoC^2-\frac{1}{4})^{k-1}$. Let us multiply by the equality $\Omega_\fosp= \D_\C a_1 + a_1\D_\C + \rhoC^2 - \frac{1}{4}$. We now have that $\Omega_\fosp^k$ is equal to: 
\[
   \left(\D_\C a_{k-1} + a_{k-1}\D_\C + (\rhoC^2-\tfrac{1}{4})^{k-1}\right)\left(  \D_\C a_1 + a_1\D_\C + \rhoC^2 -\frac{1}{4}\right )  
\]
so that
\[
\Omega_\fosp^k =  \D_\C a_k + a_k\D_\C + (\rhoC^2-\frac{1}{4})^k
\]
where $a_k = a_{k-1} (\rhoC^2- \frac{1}{4}) + a_1  (\rhoC^2- \frac{1}{4})^{k-1} + 2  a_{k-1}\D_\C a_1 + a_{k-1}(1-\epsilon(\rhoC))\rhoC\D_\C  $. 
We have proved that, if $w_0 \neq (-1)_V$ then $\zeta_\C$ satisfies condition (\ref{e:zetaC}).   

If $w_0 = (-1)_V$, then the generator of $\grZ(\tama)$ is equal to $\D\wlong$. Using the definition (\ref{d:Dirac_C}) of $\D_\C$ we have the equality
\[\bbs =\D\wlong = -\rho(\C)\wlong + \D_\C\wlong.
\]
Here $-\rho(\C)\wlong \in \ugZ(\tilde{W}_-)$.
We conclude that if $\zeta_\C$ exists it would be such that $\zeta_\C(\D\wlong) = -\rhoC\wlong \in \ugZ(\bbc\tilde{W}_-)$. To check that $\zeta_\C$ can be extended to a homomorphism with property (\ref{e:zetaC}) is identical to the part above when $w_0$ was not equal to $(-1)_V$ and we omit the details. 
\end{proof}
\subsection{Unitary structures}
In this section we define the notion of $\bullet$-Hermitian, introduce $\D_\C$-cohomology and prove that, when non-zero, the $\D_\C$-cohomology dictates that the central character of a $\tama$-module $(\pi,X)$ matches a character of any $\tilde{W}$ submodule of the given cohomology. 
\begin{definition}
Let $(\pi,X)$ be an $\tama$-module. A sequilinear form $(-,-)_X$ on $X$ is called $\bullet$-hermitian if 
\[
( \pi(\eta)x_1,x_2)_X = ( x_1,\pi(\eta^\bullet) x_2)_X \text{ for every } x_1,x_2 \in X, \eta \in \tama.
\]
Furthermore, $(\pi,X)$ is called $\bullet$-unitary if there exists a positive definite $\bullet$-Hermitian form on $X$.

\end{definition}

\begin{remark}
If $(\pi,X)$ is a module restricted from a $\bullet$-Hermitian $\rca_{t,c}\otimes \clif$-module then $(\pi,X)$ is also $\bullet$-Hermitian.
\end{remark}

\begin{proposition}
If $(\pi,X)$ is a $\bullet$-Hermitian $\tama$-module, then the operators $\pi(\D)$ and $\pi(\D_\C)$, for admissible $\C$, are self-adjoint. Furthermore, if $(\pi,X)$ is $\bullet$-unitary, then 
\[
(\D_\C^2(x),x)_X \geq 0
\]
for all $x \in X$.
\end{proposition}

\begin{proof}
For $x \in X$, since $(-,-)_X$ is positive definite, $(\D_\C(x),\D_\C(x))_X \geq 0$. Using that $\D_\C$ is self adjoint proves that $(\D_\C^2(x),x)_X \geq 0$.
\end{proof}

\begin{definition}\label{d:DiracCoh}
Let $(\pi,X)$ be an $\tama$-module and $\C$ be an admissible element in $Z^\epsilon(\bbc\tilde{W}_-)$. The {\bf $\D_\C$-cohomology} is defined by
\[
H(X,\C) = \frac{\ker(\pi\D_\C)}{\ker(\pi\D_\C)\cap \im(\pi\D_\C)}.
\]

\end{definition}

\begin{proposition}\label{p:unitarycoh}
The $\D_\C$-cohomology of $\C$ is a $\tilde W$-module. Moreover, if $X$ is a $\bullet$-unitary $\tama$-module, then $H(X,\C) = \ker(\D_\C)$.
\end{proposition}
\begin{proof} Since, the element $\D_\C$ $\epsilon$-commutes with $\bbc\tilde{W}_-$ the vector spaces $\ker(\pi\D_\C)$ and $ \im(\pi\D_\C)$ carry a $\tilde{W}$-action. Hence $H(X,\C)$ is a $\tilde{W}$-module. If $X$ is $\bullet$-Hermitian then the image and kernel of $\D_\C$ are orthogonal with respect to $(-,-)_X$ and hence $\ker(\pi\D_\C)\cap \im(\pi\D_\C)={0}$
\end{proof}

\begin{definition}
Let $\C$ be an admissible element and $\zeta_\C:\grZ(\tama)\to \ugZ(\bbc\tilde{W})$ be the homomorphism of Theorem \ref{t:Vogan}. For any irreducible $\tilde{W}$-representation $\tilde{\tau}$, define the homomorphism $\chi_{\tilde{\tau}}: \grZ(\tama)\to\bbc$ via 
\[
\chi_{\tilde{\tau}}(z) = \frac{1}{\dim\tilde{\tau}}\Tr( \tilde{\tau}(\zeta_\C(z)) ),
\]
for any $z$ in $\grZ(\tama)$.
\end{definition}

\begin{theorem}\label{t:CentChar}
Let $\C$ be an admissible element, $\tilde{\tau}$ be an irreducible $\bbc\tilde{W}_-$ representation and $(\pi,X)$ be an $\tama$-module with central character $\chi$. Suppose that
\[
\Hom_{\tilde{W}}(\tilde{\tau},H(X,\C))\neq 0.
\]
Then, $\chi = \chi_{\tilde{\tau}}$.
\end{theorem}

\noindent The proof of this theorem, given Theorem \ref{t:Vogan}, is identical to \cite[Theorem 5.11]{CDM22} and \cite[Theorem 4.5]{BCT12}. We refer the reader to \cite{CDM22} if they wish to see details. 
We can conclude that if the $\D_\C$ cohomology of $X$ is non-zero then we are able to describe the central character $\chi$ using the $\tilde{W}$ module structure of $X$. 

\section{Admissible elements}

The final part of this paper compiles descriptions of admissible elements and a criterion for non-zero $\D_\C$-cohomology of $\bullet$-unitary modules.

\subsection{The \texorpdfstring{$\epsilon$}{epsilon}-centre of \texorpdfstring{$\bbc\tilde{W}_-$}{W-}}
Before describing the admissible elements we start with a description of the $\epsilon$-centre $Z^\epsilon(\bbc\tilde{W}_-)$. We will break this section into two parts. Depending on whether $d$ is even or odd. Throughout we shall need to explore the notion of conjugacy classes ``splitting'' in different contexts (see \cite{St89}), so we formalize what we need below. It is worth noting that we consider $d \in \bbz$ to be even and odd as well as an independent notion of $\tilde{g} \in \tilde{W}$ even or odd. Recall that $\bbc \tilde{W}$ and $\bbc\tilde{W}_-$  are $\bbz_2$ graded algebras. Homogenous elements $\tilde{g}$ in $\bbc\tilde{W}$ and $\bbc\tilde{W}_-$ are even or odd when their grading is $0$ (respectively $1 \in \bbc_2$).

\begin{definition}
Given $g\in W$, we say that the conjugacy class $C(g)$ of $g$ splits in $\tilde W$ if $p^{-1}(C(g))$ consists of two conjugacy classes in $\tilde{W}$. The conjugacy class $C(g)$ does not split when $p^{-1}(C(g))$ is a single conjugacy class. Furthermore, considering the subgroup $\tilde{W}_{\overline{0}}\subseteq \tilde{W}$, we say that the conjugacy class $C(\tilde{g})$ of $\tilde{g}$ in $\tilde{W}$ splits in $\tilde{W}_{\overline{0}}$ if the conjugacy class breaks into more conjugacy classes in the subgroup $\tilde{W}_{\overline{0}}$.
\end{definition}

Recall that $\epsilon$ is uniformly $1$ when $d$ is odd. 
\begin{proposition}
Let $d$ be odd. Then, the $\epsilon$-centre of $\bbc\tilde{W}_-$ is equal to its ungraded centre. Furthermore, $Z^\epsilon(\bbc\tilde{W}_-)$ is spanned by elements of the form
\[
\{\frac{1-\theta}{2} T_{\tilde{g}} = \sum_{\tilde{w} \in \tilde{W}} \rho(\tilde{w}^{-1} \tilde{g} \tilde{w}) \in \bbc\tilde{W}_- : \tilde{g} \in \tilde{W}\}.
\]\end{proposition}
\begin{proof}
The ungraded centre of the group algebra $\bbc\tilde{W}$ is spanned by conjugacy class sums; elements of the form
\[
T_{\tilde{g}} = \sum_{\tilde{w} \in \tilde{W}} \tilde{w}^{-1} \tilde{g} \tilde{w} \in \bbc\tilde{W}
\]
for all $\tilde{g} \in \tilde{W}$. The ungraded centre of $\bbc\tilde{W}$ projects onto $\ugZ(\bbc\tilde{W}_-)$ and $\frac{1-\theta}{2}T_{\tilde{g}}$ is non-zero exactly when the
elements in $p^{-1}(p(\tilde{g}))=\{\tilde{g}, \theta \tilde{g}\}$ are in different conjugacy classes of $\tilde{W}$.
\end{proof}
For the rest of this section, we assume $d$ is even. Recall that we have $\epsilon(\tilde{w})= (-1)^{|\tilde{w}|}$. To study this case, we need to introduce a new subalgebra of $\bbc\tilde{W}$. 
\begin{definition}
Let us the define the $\theta$-centre of $\bbc\tilde{W}$,
\[
Z^\theta(\bbc\tilde{W}) = \{ a \in \bbc\tilde{W} |  a \tilde{w}= \theta^{|\tilde{w}|}\tilde{w}a \text{ for all } \tilde{w} \in \tilde{W}\},
\]
 where $|\tilde{w}|$ is the parity of $\tilde{w}$.
\end{definition}

\begin{proposition}
The $\theta$-centre of $\bbc\tilde{W}$ is spanned by elements of the form
\[
T^\theta_{\tilde{g}} =\sum_{\tilde{w} \in \tilde{W}}\theta^{|\tilde{w}|} \tilde{w}^{-1} g \tilde{w}   =  \sum_{\tilde{w} \in \tilde{W}_{\overline{0}}} \tilde{w}^{-1} \tilde{g} \tilde{w}+ \theta  \sum_{\tilde{w} \in \tilde{W}_{\overline{1}}} \tilde{w}^{-1} \tilde{g} \tilde{w}
\]
for all $\tilde{g} \in \tilde{W}$.
\end{proposition}

\begin{proof}
Given any $\tilde{g}$, then $T^\theta_{\tilde{g}}$ is in $Z^\theta( \bbc\tilde{W})$. For any $a \in Z^\theta (\bbc\tilde{W})$ that has a non-zero coefficient of $\tilde{g}$, then there exists a non-zero scalar $x$ such that $a - xT^\theta_{\tilde{g}}$ is $\theta$-central with no coefficient of $\tilde{g}$. Continuing the process shows that $a$ is in the space spanned by $T^\theta_{\tilde{g}}$. 
\end{proof}

\begin{theorem}
When $d$ is even then $\epsilon(\tilde{w}) = (-1)^{|\tilde{w}|}$ and the $\epsilon$-centre of $\bbc\tilde{W}_-$ is the projection of the $\theta$-centre of $\bbc\tilde{W}$:
\[
Z^\epsilon(\bbc\tilde{W}_-) = \rho(Z^\theta( \bbc\tilde{W})).
\]
\end{theorem}
\begin{proof}
Recall from (\ref{e:idemptildeW}) that the algebra $\bbc\tilde{W}$ comes equipped with the idempotents $\theta_{\pm} := \tfrac{1}{2}(1\pm \theta)$ and that the homomorphism $\rho$ of (\ref{e:rho}) satisfies $\rho(\theta) = 1\otimes -1$. We will write here $\rho(\theta) = -1$. Furthermore,
let $[-,-]_\theta$ denote the product $[a,b] = ab - \theta^{|b|}ba$, defined for any $\bbz_2$-homogeneous elements $a,b\in \bbc\tilde{W}$. Note that $a\in Z^\theta(\bbc\tilde W)$ if and only of $[a,\tilde w]_\theta = 0$ for all $\tilde w\in \tilde W$. All that said, we shall show, by double inclusion, that
\[
\rho(Z^\theta(\bbc\tilde W)) = Z^\epsilon(\rho(\bbc\tilde W)).
\]

Given $a\in Z^\theta(\bbc\tilde W)$, note that for any $\tilde w\in \tilde{W}$
\[
0 = \rho([a,\tilde w]_\theta) = \rho(a)\rho(\tilde w) - \rho(\theta)^{|\tilde w|} \rho(\tilde w)\rho(a)
= \rho(a)\rho(\tilde w) - (-1)^{|\tilde w|} \rho(\tilde w)\rho(a),
\]
showing that $\rho(a)$ is $\epsilon$-central. Conversely, let $\rho(a) \in Z^\epsilon(\rho(\bbc\tilde W))$. Note that this implies that, for all $\tilde w \in \tilde W$, we have
\[
0 = \rho(a)\rho(\tilde w) - (-1)^{|\tilde w|}\rho(\tilde w)\rho(a) = \rho([a,\tilde w]_\theta) = \rho([\theta_-a,\theta_-\tilde w]_\theta),
\]
since $\rho(\theta_-)=1$. It follows that $[\theta_-a,\theta_-\tilde w]_\theta$ is in $\theta_+(\bbc\tilde W)$, the kernel of $\rho$, and hence, we must have $[\theta_-a,\theta_-\tilde w]_\theta\in \bbc\tilde W_+\cap \bbc\tilde W_- = \{0\}$, for any $\tilde w\in\tilde W$. Now choose any element $a'\in Z^\theta(\bbc \tilde W)$ and define $b = \theta_+a' + \theta_-a \in \bbc\tilde W$. Then, for any $\tilde w\in \tilde W$, we have
\begin{align*}
 [b,\tilde w]_\theta &=  [\theta_+a' + \theta_-a,\theta_+\tilde w + \theta_-\tilde w]_\theta\\
 &=  [\theta_+a',\theta_+\tilde w]_\theta+[\theta_-a,\theta_-\tilde w]_\theta\\
 &= 0 ,
\end{align*}
where we used $\theta_+\theta_- = \theta_-\theta_+=0$, $[\theta_+a',\theta_+\tilde w]_\theta = \theta_+[a',\tilde w]_\theta =0$ (as $a'\in Z^\theta(\bbc \tilde W)$) and also $[\theta_-a,\theta_-\tilde w]_\theta = 0$ (from the assumption we made). Hence, $b\in Z^\theta(\bbc\tilde W)$ and $\rho(b) = \rho(a) \in Z^\epsilon(\rho(\bbc\tilde W))$, finishing the proof.
\end{proof}

\noindent In particular (when $d$ is even), the $\epsilon$-centre of $\bbc\tilde{W}_-$ is spanned by elements of the form 
\[
T^{(-1)}_{\tilde{g}} = \frac{1-\theta}{2}T^\theta_{\tilde{g}} = \rho(T^\theta_{\tilde{g}})= \sum_{\tilde{w} \in \tilde{W}}(-1)^{|\tilde{w}|} \rho(\tilde{w}^{-1} \tilde{g} \tilde{w}) \in \bbc\tilde{W}_-.
\]
These elements are signed sums of conjugacy class elements and are homogenous. We now describe precise conditions when these spanning elements are nonzero.

\begin{lemma} \label{l:oddadm1}
If $\tilde{g}$ is odd then $T^{(-1)}_{\tilde{g}} =0$.
\end{lemma}
\begin{proof}
 If $\tilde{g}$ is taken to be odd then $\rho(\tilde{g})^{-1} T^{(-1)}_{\tilde{g}} \rho(\tilde{g}) = - T^{(-1)}_{\tilde{g}}$. Isolating the $\rho(\tilde{g})$ coefficient in $T^{(-1)}_{\tilde{g}}$ shows that it must be zero and hence $T^{(-1)}_{\tilde{g}}$ is zero for every odd element $\tilde{g}$.
\end{proof}

\begin{lemma} \label{l:splitsplit}
Suppose that $\tilde{g} \in \tilde{W}$ is even, let $C(\tilde{g}) = \{ \tilde{w }\in \tilde{W}: \tilde{w} = \tilde{h}^{-1}\tilde{g}\tilde{h}, \tilde{h} \in \tilde{W}\}$, then the element $T^{(-1)}_{\tilde{g}}$ is non zero if and only if the conjugacy class $C(\tilde{g})$ splits into two conjugacy classes in $\tilde{W}_{\overline{0}}$.

\end{lemma}

\begin{proof}
First assume that $C(\tilde{g})$ splits in $\tilde{W}_{\overline{0}}$ and let $g = p(\tilde{g})$.
We divide the proof of this implication depending on whether or not $C(g)$ splits in $\tilde{W}$. First assume that $C(g)$ does not split in $\tilde{W}$ but $C(\tilde{g})$ splits in $\tilde{W}_{\overline{0}}$, then the conjugacy class of $\tilde{g}$ in $\tilde{W}$ is $C(\tilde{g}) = C^+(\tilde{g}) \cup C^-(\tilde{g})$, where $C^\pm(\tilde{g})$ are conjugacy classes in $\tilde{W}_{\overline{0}}$ and $C^\pm(\tilde{g})= \theta C^\mp(\tilde{g})$. Then 
\[
T^\theta_{\tilde{g}} =\sum_{\tilde{h} \in C^+(\tilde{g})}\tilde{h} + \theta\sum_{\tilde{h} \in C^-(\tilde{g})} \tilde{h}= 2\sum_{\tilde{h} \in C^+(\tilde{g})}\tilde{h}
\]
and the projection $T^{(-1)}_{\tilde{g}} = 2\sum_{\tilde{h} \in C^+(\tilde{g})}\rho(\tilde{h})$ is non-zero in $Z^\epsilon(\bbc\tilde{W}_-)$.
 
 Now suppose that $C(g)$ splits in $\tilde{W}$ into $C(\tilde{g})$ and $C(\theta\tilde{g}) = \theta C(\tilde{g})$, where $\{\tilde{g},\theta\tilde{g}\}= p^{-1}(g)$.  Then, by the assumptions of the lemma each of these conjugacy classes splits into two conjugacy classes $C^\pm(\tilde{g})$ and $C^\pm(\theta\tilde{g})$ in $\tilde{W}_{\overline{0}}$. Then $T^\theta_{\tilde{g}} = \sum_{\tilde{h} \in C^+(\tilde{g})}\tilde{h}  + \theta\sum_{\tilde{h} \in C^+(\theta\tilde{g})}\tilde{h} $ and 
 \[
 T^{(-1)}_{\tilde{g}} =  \sum_{\tilde{h} \in C^+(\tilde{g})}\rho(\tilde{h}) - \sum_{\tilde{h} \in  C^+(\theta\tilde{g})}\rho(\tilde{h}),
 \] 
 which is non-zero since $\rho(C^+(\tilde{g}))$ and $\rho(C^+(\theta\tilde{g}))$ are sums over linearly independent elements in $\bbc\tilde{W}_-$.

Suppose now, that $C(\tilde{g})$ in $\tilde{W}$ remains a conjugacy class in $\tilde{W}_{\overline{0}}$ then $T^\theta_{\tilde{g}} = \sum_{\tilde{w} \in \tilde{W}_{\overline{0}}} \tilde{w}^{-1} \tilde{g} \tilde{w}+\theta  \sum_{\tilde{w} \in \tilde{W}_{\overline{1}}} \tilde{w}^{-1} \tilde{g} \tilde{w}$ and 
 \[
 T^\theta_{\tilde{g}} = (1+\theta) \sum_{\tilde{w} \in \tilde{W}_{\overline{0}}} \tilde{w}^{-1} \tilde{g} \tilde{w}.
 \]
 Hence $T^{(-1)}_{\tilde{g}} = \frac{1-\theta}{2}T^\theta_{\tilde{g}} = 0$  because $(1+\theta) (1-\theta) = 0$. 
\end{proof}

\subsection{The real subspace of admissible elements}
We have now described the $\epsilon$-centre of $\bbc\tilde{W}_-$. The following two propositions (split into $d$ even and $d$ odd), describe the set $\fA$ of admissible elements (Definition \ref{d:Dirac_C}). That is, the elements that are $\epsilon$-central and self-adjoint.



\begin{proposition}\label{l:oddadm}
Let $d$ be odd. In this case $\epsilon(w) = 1$ for all $w$. Then $\fA =  \ugZ_{\overline{0}} (\bbr\tilde{W}_-) + i  \ugZ_{\overline{1}}(\bbr\tilde{W}_-)$.

\end{proposition}
\begin{proof}
If $d$ is odd then the set of admissible elements $\fA$ is the intersection of $\ugZ(\bbc\tilde{W}_-)$ and the self adjoint elements in $\bbc\tilde{W}_-$. Fix a basis $B = \{\prod s_\alpha\otimes \alpha/{|\alpha|}: w \in W, w = \prod s_\alpha \text{ is a fixed word for } w \textup{ with }\alpha \textup{ simple}\}$ of $\rho(\bbc\tilde W) \cong \bbc\tilde W_-$. Note that $B$ has size $|W|$. Since $(\alpha/{|\alpha|})^\bullet = -\alpha/{|\alpha|}$ and $B = B_{\overline 0}\cup B_{\overline 1}$ consists of homogeneous elements then $B_{\overline{0}}^\bullet = B_{\overline{0}}$ and $B_{\overline{1}}^\bullet = -B_{\overline{1}}$. We can conclude that $(iB_{\overline{1}})^\bullet = iB_{\overline{1}}$. 
Given that $B_{\overline{0}} \cup iB_{\overline{1}}$ is a $\bullet$-invariant basis and, for any admissible element, its odd and even part are admissible, then the result follows.
\end{proof}

The ungraded centre of the group algebra $\bbc\tilde{W}$ is spanned by the conjugacy class sums. The algebra $\bbc \tilde{W}_-$ is a quotient and a subalgebra of $\bbc\tilde{W}$. The ungraded centre of $\bbc\tilde{W}_-$ is exactly $\frac{1-\theta}{2}\ugZ(\bbc\tilde{W})$.
\begin{lemma}\label{l:noodd}
The space $Z^\epsilon_{\overline{1}}(\bbc\tilde{W}_-)$ is zero.
\end{lemma}

\begin{proof}
 The space $\grZ_1(\bbc \tilde{W})$ is the span of the elements $T^{(-1)}_{\tilde{g}}$ for odd $\tilde{g}$. However, Lemma \ref{l:oddadm1} shows that these elements are always zero.
\end{proof}

\begin{proposition}
Let $d$ be even and let $\fA$ be the set of the admissible elements in $\bbc\tilde{W}_-$. In this case $\epsilon(w) = -1$ for all odd $w$. Then $\fA =  Z^\epsilon_{\overline{0}} (\bbr\tilde{W}_-) $.

\end{proposition}

\begin{proof}
The same mechanism as the proof of Lemma \ref{l:oddadm} apply here. Noting that the $\epsilon$-centre splits into homogenous parts; $Z^\epsilon(\bbc\tilde{W}_-) = \anZ_{\overline{0}}(\bbc \tilde{W}_-) \oplus \grZ_{\overline{1}}(\bbc\tilde{W}_-) $. The Lemma follows from intersecting with the real subspace of self adjoint elements and Lemma \ref{l:noodd} which states that $\grZ_{\overline{1}}(\bbc\tilde{W}_-) = Z^\epsilon_{\overline{1}}(\bbc\tilde{W}_-)=0$.\end{proof}

\begin{theorem}\label{t:oddadm}
Let $d$ be odd. Let $S_{\mathrm{split}}$ be the a set of representatives of conjugacy classes in $W$ which split in $\tilde{W}$. That is,
\[S_{\mathrm{split}} = \{\text{Representatives of conjugacy classes}\} \cap\{g \in W: C(g) \text{ splits in } \tilde{W}\}
\]
Then the admissible elements in $\bbc\tilde{W}_-$ is equal to the real span of the linearly independent set
\[
\{\sqrt{-1}^{|\tilde{g}|}\frac{1-\theta}{2}T_{\tilde{g}}:  g \in S_{\mathrm{split}} \}.
\]
\end{theorem}
\begin{proof}
By Lemma \ref{l:oddadm} the admissible elements are the real spans of the even ungraded centre plus the imagine span of the odd ungraded centre. The proof follows from the observation that the ungraded centre $\ugZ\bbc\tilde{W}_-$ is spanned by $\{\frac{1-\theta}{2}T_{\tilde{g}}:\tilde{g} \in S_{\mathrm{split}}\}$.
\end{proof}
When $d$ is odd the real dimension of the admissible elements is equal to the number of irreducible, genuine projective representation of $W$. That is, the representations of $\tilde{W}$, which do not factor through $\bbc W$.

\begin{theorem}\label{t:evenadm}
Let $d$ be even. Let $S^\epsilon_{\mathrm{split}}$ be the a set of representatives of conjugacy classes in $\tilde{W}$ which split in $\tilde{W}_{\overline{0}}$. Then the admissible elements in $\bbc\tilde{W}_-$
has real  basis
\[\{T^{(-1)}_{\tilde{g}}: \tilde{g} \in S^\epsilon_{\mathrm{split}}\}.
\]
\end{theorem}

\begin{proof}
Similar to the proof of Theorem \ref{t:oddadm}. In this setting the admissible elements are the real spans of the even $\epsilon$-central elements plus the imaginary spans of the odd $\epsilon$-central elements, of which there are none (by Lemma \ref{l:noodd}). The even $\epsilon$-central elements are spanned by 
\[
\{T^{(-1)}_{\tilde{g}}: C(\tilde{g}) \subset \tilde{W} \text{ splits into two conjugacy classes in   } \tilde{W}_{\overline{0}}\}.
\]
The Theorem follows. 
\end{proof}


 For $n\in \bbz$, we state a couple of theorems describing when conjugacy classes of the symmetric group $S_n$ split when considered in $\tilde{S}_n$ and $\tilde{A}_n$. We then draw a corollary (\ref{c:oddsymmadm}) from Theorems \ref{t:oddadm} and \ref{t:evenadm} describing the admissible elements in $\bbc\tilde{W}_-$, for $W=S_n$ acting on any vector space $V \cong\bbc^d$. Note that the $W$-module $V$ need not be irreducible or essential. 
 \begin{theorem} \cite[p. 1721]{Sc11}
 The conjugacy classes of $S_n$ which split into two conjugacy classes of $\tilde{S}_n$ are precisely those with all cycles of odd length or those that are odd and have distinct cycles .  
 \end{theorem}

  \begin{theorem}\cite{Sc11} \cite[Theorem 2.7]{St89}
 Let $\lambda$ be an even partition of $n$. The $\tilde{S}_n$ conjugacy classes $C_\lambda$ (or $C_\lambda^{\pm})$ if already split) split into two $\tilde{A}_n$ conjugacy classes iff $\lambda \in DP_n^+$. Here $DP_n^+$ is the set of distinct partitions of $n$ which are even. 
 
 \end{theorem}
 
 \begin{corollary}
\label{l:typeAex}
Let $W=S_n$ and let $\{g\}$ be the set of elements in $S_n$ associated to an even partition $\lambda$ which has distinct cycles. Then $C(\tilde{g})$ splits in $\tilde{A}_n = (\tilde{S}_n)_{\overline{0}}$. Then from Lemma \ref{l:splitsplit}, $T^{(-1)}_g \neq 0$ and is an admissible element. 
\end{corollary}

\begin{corollary}\label{c:oddsymmadm}
Let $d$ be odd and $W = S_n$ acting on $V \cong \bbc^d$, then the admissible elements in $\bbc\tilde{W}_-$ is equal to the real span of the set \[\{T_{\tilde{g}}: g \text{ has partition with no even cycles}\}.\]
Let $d$ be even and $W = S_n$, then the admissible elements in $\bbc\tilde{W}_-$ is equal to the real span of the set $\{T^{(-1)}_{\tilde{g}}: g \text{ has partition } \lambda \in DP_n^+\}$.
\end{corollary}





\subsection{Non-zero \texorpdfstring{$\D_\C$}{}-cohomology}

\begin{proposition}
Suppose that $(\pi,X)$ is $\bullet$-unitary. Since $\Omega_\fosp$ and  $\C$ are self adjoint they must have positive real eigenvalues. 
\end{proposition}

\begin{proposition}\label{p:nonzerocoh}
Let $(\pi,X)$ be a $\bullet$-unitary module for $\tama$. Suppose that there exists an admissible element $\C$ such that $\pi(\rhoC) \neq 0$. Then, there exists an admissible $\C \in \bbc \tilde{W}_-$ such that $H(X,\C) \neq 0$.
\end{proposition}

\begin{proof}
Since $X$ is Hermitian then $H(X,\C) = \ker(\D_C) = \ker (\D_C^2)$. We study the kernel of the operator $(\D_\C -2\rho(\C)))\D_\C =\Omega_\fosp +\frac14 + \rho(\C)^2$. The elements $\Omega_\fosp$ and $\rho(\C)$ commute, hence have simultaneous eigenvalues. Given $\pi(\rhoC) \neq 0$ it has a positive real eigenvalues. One can modify $\D_\C$ to $\D_{\lambda\C}$ to ensure that 
\[
\Omega_\fosp +\frac14 + \lambda^2\rho(\C)^2
\]
has a non-zero kernel. This then proves that there exists a non-zero kernel for the operator $(\D_\C -(1+\epsilon(\rho(\C))\rho(\C))\D_\C$. It is impossible for both $\D_\C$ and $\D_\C-(1+\epsilon(\rho(\C))\rho(\C) = \D_{-(\epsilon(\rho(\C))\C}$ to have zero kernel.   Hence there exists an admissible element which gives non-zero $\D_\C$-cohomology.
\end{proof}

\noindent If $W=S_d$, then $\ugZ_{\overline{0}}(\bbr\tilde{W}_-)$ is the symmetric polynomials in the squares of the Jucys-Murphy elements. Furthermore, $\ugZ_{\overline{1}}(\bbr\tilde{W}_-)$ is the projection of odd central elements in $\bbc\tilde{W}$, which since $\bbc\tilde{W}$ is a group algebra can be written as conjugacy class sums. It is shown in \cite{SV08} and \cite{Ca19} that these central elements act by a positive integer or half integer on irreducible modules. This observation, in combination with Proposition \ref{p:nonzerocoh} proves that if $W=S_d$, for every $\bullet$-Hermitian module $X$, there exists a $\D_\C$ with non-zero $\D_\C$-cohomology. 

\begin{example}

Let $W$ contain $(-1)_V$, let us show that the action of the admissible element $w_0\otimes \Gamma$ acts on $\tama$ representations with nonzero eigenvalues. We note that $w_0\otimes \Gamma$ is invertible and self adjoint. Therefore, on any $\bullet$-Hermitian module $(\pi,X)$ the eigenvalues of $\pi(w_0\otimes \Gamma)$ must be positive reals. Therefore, if $d$ is even and $(-1)_V \in W$, any $\bullet$-Hermitian $\tama$-module $(\pi,X)$ has non-zero $\Gamma$-cohomology for the correct choice of $\C$.
\end{example}

\section*{Acknowledgements}
We thank the Department of Mathematics, University of Manchester and both the Department of Applied Mathematics, Computer Science and Statistics, and the Department of Electronics and Information Systems at Ghent University for their hospitality during the research visits throughout the preparation of this manuscript. This research was supported by the Heilbronn Institute for Mathematical Research (K.C.), the special research fund (BOF) from Ghent University [BOF20/PDO/058] (M.D.M.), and by a postdoctoral fellowship, fundamental research, of the Research
Foundation – Flanders (FWO) [grant number 12Z9920N] (R.O.). 
This support is gratefully acknowledged.

\bibliographystyle{abbrv}

\bibliography{References}

\small 

\noindent Kieran Calvert, \texttt{kieran.calvert@manchester.ac.uk}\\
Department of Mathematics, 
University of Manchester,
Alan Turing Building,
Oxford Road,
Manchester, UK.\\

\noindent Marcelo De Martino,
\texttt{marcelo.goncalvesdemartino@ugent.be}\\
Department of Electronics and Information Systems, 
Ghent University, 
Krijgslaan 281, Building S8, 9000 Gent, Belgium.\\

\noindent Roy Oste, 
\texttt{Roy.Oste@UGent.be}\\
Department of Applied Mathematics, Computer Science and Statistics, 
Ghent University, 
Krijgslaan 281, Building S9, 9000 Gent, Belgium.
\end{document}